\documentclass[11pt]{article}

\usepackage{amssymb, amsmath, amsthm}
\usepackage{verbatim}
\usepackage{geometry}
\usepackage[all]{xy}
\usepackage{stmaryrd}
\usepackage{mathtools}
\usepackage{cite}
\usepackage[colorlinks=true,urlcolor=blue]{hyperref}

\usepackage{graphicx}
\DeclareGraphicsExtensions{.pdf,.png,.jpg}

\newcounter{reminder}
\newcommand{\reminder}[1]{\marginpar{\stepcounter{reminder}\textcolor{red}{Note~\arabic{reminder}}}$\langle$~{\sf #1 }~$\rangle$}

\newtheorem{theorem}{Theorem}[section]
\newtheorem{lemma}[theorem]{Lemma}
\newtheorem{proposition}[theorem]{Proposition}
\newtheorem{corollary}[theorem]{Corollary}
\newtheorem{definition}[theorem]{Definition}

\newcommand{\calP}{\mathcal P}
\newcommand{\eps}{\epsilon}

\newcommand{\mix}{\diamond}
\newcommand{\comix}{\boxempty}

\theoremstyle{remark}

\title{The ranks of alternating string C-groups}

\author{Mark Mixer}

\begin{document}

\maketitle

\begin{abstract}
	
		In this paper, string C-groups of all ranks $3 \leq r <  \frac{n}{2}$ are provided for each alternating group $A_n $, $n \geq 12$.  As the string C-group representations of $A_n$ have also been classified for $n \leq 11$, and it is known that larger ranks are impossible, this paper provides the exact values of $n$ for which $A_n$ can be represented as a string C-group of a fixed rank.
		\end{abstract}

\section{Introduction}

In Problem 32 of~\cite{polytope-problems}, Hartley asks ``Find regular, chiral, or other polytopes whose automorphism
groups are alternating groups $A_n$. In particular, given a rank $r$, for which $n$ does
$A_n$ occur as the automorphism group of a regular or chiral polytope of rank $r$?"
In~\cite{HighestAn}, the maximum achievable rank for each group $A_n$ was found in the regular case.
In this paper we finish the solution to Hartley's question in the regular case, by finding string C-groups (regular polytopes) of all achievable ranks $r$ for each $n$.

The paper is organized as follows. In section~\ref{sec2} we briefly outline any necessary definitions and background.   Sections~\ref{sec3} contains many families of string C-groups that are needed to prove the main theorem. Sections~\ref{sec4} and \ref{sec5} consider string C-groups of rank at least seven, for odd and even $n$ respectively.  In Section~\ref{sec6}, we provide the string C-groups with ranks less than or equal to six for all possible $n$.  Finally, in section~\ref{sec7}, we summarize the main theorem.


\section{Background and Basic Notions \label{sec2}}

The automorphism group of an abstract regular polytope, along with a distinguished set of generators $\{ \rho_0,\ldots,\rho_{r-1} \}$, is called a rank $r$ string C-group.  In general, we say that a group $\Gamma$ is a rank $r$ {\em string group generated by involutions} (or an {\em sggi} for short) if $\Gamma$ is generated by $\{ \rho_0,\ldots,\rho_{r-1} \}$ which satisfy the following conditions.  
\begin{equation}
\label{coxrel}
(\rho_i \rho_j)^{p_{ij}}=\eps \quad (0\leq i,j \leq r-1),
\end{equation}
where $p_{ii}=1$ for all $i$, $2 \leq p_{ji} = p_{ij} $ if $j=i-1$, and  
\begin{equation}
\label{string}
p_{ij}=2 \textrm{ for } |i-j| \geq 2.
\end{equation}

Moreover, if $\Gamma$ has the following \textit{intersection property}, then it is considered to be a {\em string C-group}.
\begin{equation}
\label{intprop}
\langle \rho_i \mid i\in I \rangle \cap \langle \rho_i \mid i\in J \rangle 
= \langle \rho_i \mid i\in  I  \cap J \rangle \,\textrm{ for }\, I,J \subseteq \{0,\ldots,r-1\}
\end{equation}
The automorphism group $\Gamma( \calP )$ of an abstract regular polytope $\calP$ is a string C-group, and conversely, it is known (see \cite[Sec. 2E]{arp}) that an abstract regular $n$-polytope can be constructed uniquely from any string C-group.  

We will often use the fact that not all of the intersections from Equation~\ref{intprop} need to be verified.  

\begin{proposition} \label{lem:max}
Let $\Gamma$ be a rank $r$ string group generated by involutions, and suppose that $\Gamma_0$ and $\Gamma_{r-1}$ are both string C-groups. If $\Gamma_0 \cap \Gamma_{r-1} = \Gamma_{0,r-1}$, then $\Gamma$ is a string C-group.  Moreover, if $\Gamma_{0,r-1}$ is a maximal subgroup of either $\Gamma_0$ or $\Gamma_{r-1}$ then this condition is satisfied.
\end{proposition}
\begin{proof}
This combines Proposition 2E16 of~\cite{arp} and Lemma 2.2 of~\cite{high-rank-alternating}.
\end{proof}

Let $\Gamma=\langle\rho_0,\,\ldots,\,\rho_{r-1}\rangle$ be a string group generated by involutions acting as a permutation group on a set $\{1,\,\ldots,\,n\}$.
We can construct the {\em permutation representation graph} $\mathcal{X}$ of $\Gamma$ as the $r$-edge-labeled graph with $n$ vertices, and with a single $i$-edge $\{a,\,b\}$ whenever $a\rho_i=b$ with $a < b$.   When $\Gamma$ is a string C-group that acts faithfully on $\{1, \ldots, n\}$, the graph $\mathcal{X}$ is called a \emph{CPR graph}, as defined in \cite{DanielCPR}.

If $P$ and $Q$ are string C-groups, then we say that $P$ \emph{covers} $Q$ if there is a well-defined
surjective homomorphism from $P$ to $Q$ that respects the canonical generators. In other words,
if $P = \langle \rho_0, \ldots, \rho_{r-1} \rangle$ and $Q = \langle \rho_0', \ldots,
\rho_{r-1}' \rangle$, then $P$ covers $Q$ if there is a homomorphism that sends each $\rho_i$
to $\rho_i'$. 

Given string C-groups $P$ and $Q$, the \emph{mix} of $P$ and $Q$, denoted $P \mix
Q$, is the subgroup of the direct product $P \times Q$ that is generated by the
elements $(\rho_i, \rho_i')$. This group is the minimal string group generated by involutions that covers both $P$ and $Q$ - where again, we only consider homomorphisms that respect the generators; see \cite[Section 5]{mixing-and-monodromy} for more details.

It is possible to mix a rank $r$ string C-group $P$ with a rank $s$ string C-group $Q$.  In particular we often mix a string C-group with the automorphism group of an edge $e$ (which is a rank 1 regular polytope). To do so, we take
$e = \langle \rho_0, \ldots, \rho_{r-1} \rangle$ with defining relations $\rho_0^2 = \eps$ and
$\rho_i = \eps$ for $1 \leq i \leq r-1$, and then use the same definition as before.   In general, to mix two string C-groups of different ranks, we add trivial generators to the group of smaller rank.

The \emph{comix} of $P$ and $Q$, denoted $P \comix Q$, is the largest
string group generated by involutions that is covered by both $P$ and $Q$ \cite{var-gps}. A presentation for $P \comix Q$ can be obtained from that of $P$ by adding all of the relations of
$Q$, rewriting the relations to use the generators of $P$ instead.  The size of the comix of $P$ and $Q$ is the index of the mix in the full direct product.  Throughout the paper we will rely on the following results about mixing of string C-groups.

\begin{proposition}
\label{lem:MixMon}(Theorem 5.12 of~\cite{mixing-and-monodromy})
Suppose $P$ and $Q$ are rank $r$ string C-groups, and that $P_{r-1}$ covers $Q_{r-1}$ then $P \mix Q$ is a string C-group.
\end{proposition}

\begin{proposition}\label{lem:MixFacet}
(Theorem 7A7 of~\cite{arp})
If $P$ is a rank $r$ string C-group then $P \mix P_{r-1}$ is a string C-group.
\end{proposition}


\begin{definition}
Let $\Gamma = \langle \rho_0,\ldots,\rho_{r-1} \rangle$ be an sggi, and let $\tau$ be an involution in a supergroup of $\Gamma$ such that $\tau \not \in \Gamma$ and $\tau$ commutes with all of $\Gamma$.  For fixed $k$, we define the group $\Gamma^*= \langle \rho_i \tau^{\eta_i}\,|\, i\in \{0,\,\ldots,\,r-1\} \rangle$ where $\eta_i = 1$ if $i=k$ and 0 otherwise, the sesqui-extension of $\Gamma$ with respect to $\rho_k$ and $\tau$.
\end{definition}
A sesqui-extension of a group $\Gamma$, with respect to its first generator can be seen as a mix of $\Gamma$ with the automorphism group of an edge, and thus we have the following.

\begin{proposition}(Proposition 5.3 of~\cite{Alt2})
\label{se0}
If $\Psi$ is a sesqui-extension of a string C-group $\Gamma$ with respect to $\rho_0$, then  $\Psi$ is a string C-group.
\end{proposition}

The following lemma shows that we can often use this result in a more general setting.

\begin{lemma} 
\label{se} (Lemma 5.4 of~\cite{Alt2})
If $\Gamma=\left<\rho_i\,|\, i=0,\,\ldots,\,r-1\right>$ and $\Psi=\langle \rho_i \tau^{\eta_i}\,|\, i\in \{0,\,\ldots,\,r-1\} \rangle$ is a sesqui-extension of $\Gamma$ with respect to $\rho_k$, then:
\begin{enumerate}
\item $ \Psi \cong \Gamma$ or $\Psi \cong \Gamma\times \langle \tau\rangle\cong\Gamma\times 2$.
\item If the identity element of $\Gamma$ can be written as a product of generators involving an odd number of  $\rho_k$'s, then $\Psi \cong\Gamma\times \langle \tau\rangle$.
\item if $\Gamma$ is a finite permutation group, $\tau$ and $\rho_k$ are odd permutations, and all other $\rho_i$ are even permutations, then $\Psi \cong\Gamma$.
\item whenever $\tau \notin \Psi$, $\Gamma$ is a string C-group if and only if $\Psi$ is a string C-group.
\end{enumerate}
\end{lemma}

\begin{proposition} \label{ext}
Let $\Gamma=\left<\rho_i\,|\, i=0,\,\ldots,\,r-1\right>$ 
and $\Psi=\langle \rho_i \tau^{\eta_i}\,|\, i\in \{0,\,\ldots,\,r-1\} \rangle$ be a sesqui-extension of $\Gamma$ with respect to $\rho_k$.  If either $\Psi_0 \cong \Gamma_0$ or $\Psi_{r-1} \cong \Gamma_{r-1}$ as string C-groups, then $\Psi$ is a string C-group.
\end{proposition}

\begin{proof} 
This is a consequence of part (b) of Proposition 2E16 in~\cite{arp}.  Assume $\Psi_{r-1} \cong \Gamma_{r-1}$ as string C-groups. The intersection condition of part (b) holds as $\tau$ is not in any of the groups $\Psi_{r-1} \cap \langle \rho_k,\ldots, \rho_{r-1} \rangle$. 
\end{proof}

For any permutation group $\Gamma$ of degree $n$, we will use the notation $S_n$ to represent the full symmetric group, $A_n$ to represent the full alternating group, and $G^+$ to denote $G \cap A_n$.  If $\Gamma:=  \langle \rho_0, \ldots, \rho_{n-1} \rangle $, then for each $i$ we denote $\Gamma_i = \langle \rho_j \mid j \neq i \rangle$, where each $\Gamma_i$ is itself a string C-group.  Similarly, we will denote $\Gamma_{i,j} = \langle \rho_k \mid k \not \in \{i, j\}  \rangle$.
The {\em dual} $\Gamma^*$ of a string C-group $\Gamma$ is the group generated by the same involutions, but with the indexing reversed.

Finally, we will occasionally use the following rank reduction technique of~\cite{rank-reduction}.  We will frequently need to apply this construction to the dual of a string C-group, and then take the dual again.  When we do this, we simply call it the dual rank reduction.
\begin{proposition} \label{lem:RR}
Let $\Gamma = \langle \rho_0, \ldots \rho_{r-1} \rangle $ be a rank $r$ string C-group, where $| \rho_i \rho_{i+1} | > 2$ for all $0 \leq i \leq r-2$.   If $\rho_0 \in \langle \rho_0 \rho_2, \rho_3 \rangle$, then $\Gamma \cong \langle \rho_1, \rho_0 \rho_2, \rho_3, \ldots, \rho_{r-1} \rangle$ is a string C-group of rank $r-1$.
Furthermore, if $ \rho_2 \rho_3 $ has odd order, then this condition is satisfied.

\end{proposition}
\begin{proof}
This combines the results of Theorem 1.1 and Corollary 1.2 of~\cite{rank-reduction}.
\end{proof}


\section{Building Blocks}\label{sec3}

In this section, we provide some examples of families of string C-groups which appear as subgroups of the groups in our main theorem.

\begin{lemma}\label{lem:FL}
For each $r \geq 3$ the permutation group $FL(r,k)$ given by the following graph is a string C-group isomorphic to $S_{r+1+k}$, for all odd $k \geq 0$, and 
$$\includegraphics[scale=.5]{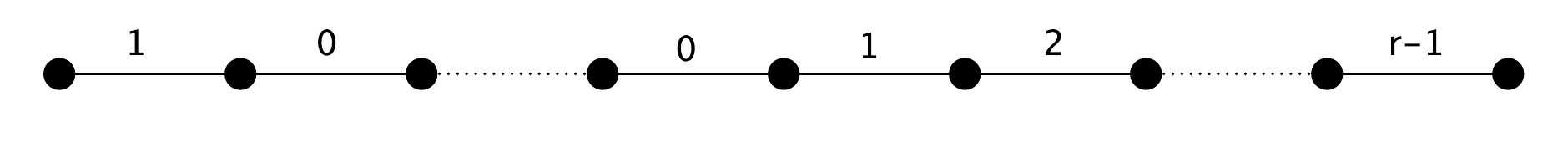}$$
for each $r \geq 3$ the permutation group $FL(r,k)$ given by the following graph is a string C-group isomorphic to $S_{r+1+k}$, for all even $k \geq 0$
$$\includegraphics[scale=.5]{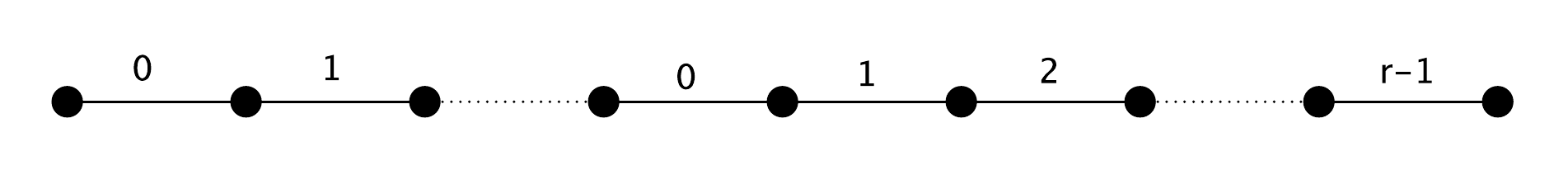}$$
\end{lemma}

We point out that there are $k+2$ edges of labels 0 or 1 in each graph.

\begin{proof}
This combines the results of Theorem 1, Theorem 2, and Lemma 21 from~\cite{high-rank-symmetric}.
\end{proof}


\begin{lemma}\label{lem:R}
For each $r \geq 5$ the permutation group $R(r,k)$ given by the following graph is a string C-group isomorphic to $S_{r+3+k}$, for all odd $k \geq 0$, and 
$$\includegraphics[width=6in]{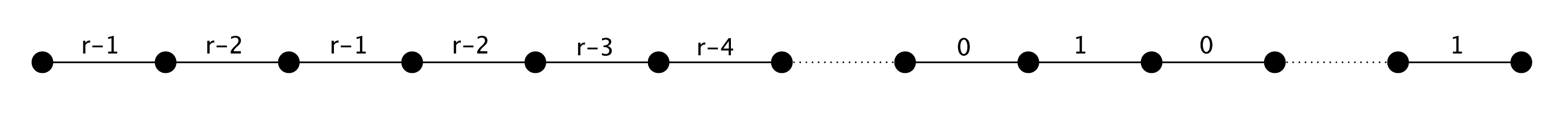}$$
for each $r \geq 5$ the permutation group $R(r,k)$ given by the following graph is a string C-group isomorphic to $S_{r+3+k}$, for all even $k \geq 0$. 
$$\includegraphics[width=6in]{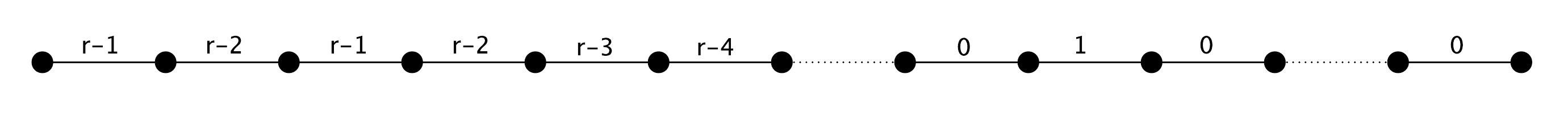}$$
\end{lemma}

\begin{proof}  
We will prove this by induction on $k$.  To clarify the notation, we point out that there are $k+2$ edges of labels 0 or 1 in each graph.  When $k=0$, $R(r,k)$ is a string C-group isomorphic to $S_{r+3}$ as it is the dual of $FL(r,2)$ from Lemma~\ref{lem:FL}.  We will assume by induction that $R(r+1,k-1)$ is a string C-group isomorphic to $S_{r+1+3+k-1}$, and then applying the rank reduction from Proposition~\ref{lem:RR} to $R(r+1,k-1)$ we get $R(r,k)$, which is thus a string C-group isomorphic to $S_{r+3+k}$.

\end{proof}


\begin{lemma}\label{lem:Sh}
For each $r \geq 4$ the permutation group $Sh(r,k)$ given by the following graph is a string C-group isomorphic to   $ (S_2 \wr S_{r+\frac{k}{2}})^+$, for all $k \equiv 2 \pmod{4}$, with $k \geq 0$.  
$$\includegraphics[scale=.5]{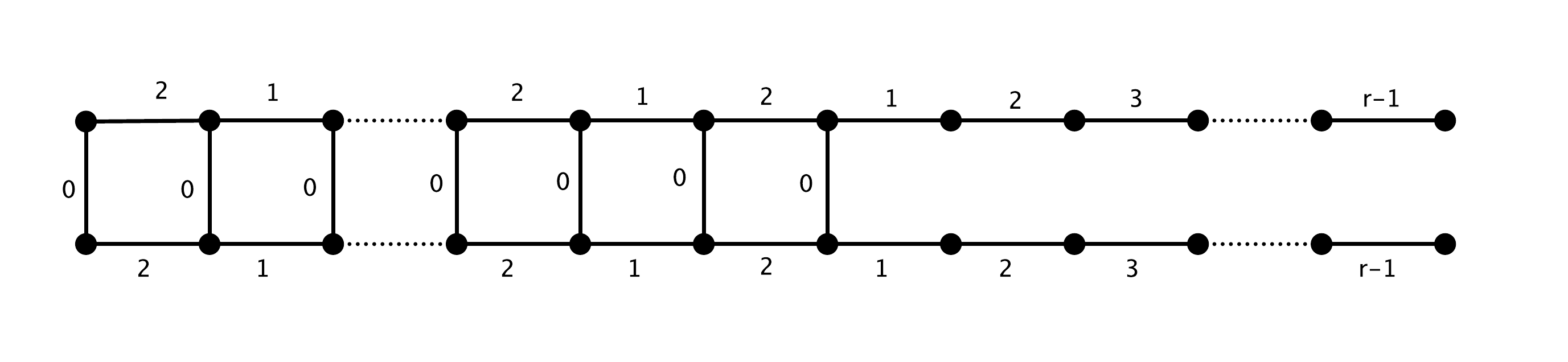}$$
\end{lemma}
Note that there are $1+\frac{k}{2}$ edges of label 0 in each such graph.  We also note that when $r=3$ this graph still provides a string C-group. 

\begin{proof}
We prove that the group is a string C-group by induction on $r$, where the proof of the base case will be nearly the same as the proof of the inductive step.
Let $\Gamma = \langle \rho_0, \ldots, \rho_{r-1} \rangle = Sh(r,k)$.  When $r=4$, the group $\Gamma_{r-1}$ is a string C-group by Theorem 4.4 of~\cite{DanielCPR}.
Furthermore, when $r$ is greater than $4$, we can assume that $\Gamma_{r-1}$ is a string C-group as $\Gamma_{r-1} = Sh(r-1,k)$.

For all $r \geq 4$, the group $\Gamma_0$ is isomorphic to the string C-group $FL(r-1,\frac{k}{2}) \diamond FL(r-1,\frac{k}{2})$ which is isomorphic to $S_{r+\frac{k}{2}}$.  It remains to show that  $\Gamma_0 \cap  \Gamma_{r-1} =  \Gamma_{0,r-1}$.

 The group $\Gamma$ is an imprimitive group with a natural block structure, having $r+\frac{k}{2}$ blocks of size two.   The generator $\rho_0$ is the only generator acting within a block, and thus $\Gamma_0$ gives the action on the blocks.  
 
Let $\alpha \in \Gamma_0 \cap \Gamma_{r-1}$.  Since $\alpha \in \Gamma_0$ it only acts on the blocks, and since $\alpha \in \Gamma_{r-1}$ it fixes the ``last block" (the block in the support of $\rho_{r-1}$).  Since $\Gamma_{0,r-1}$ gives the action on these $r+\frac{k}{2}-1$ blocks, we know that $\alpha \in \Gamma_{0,r-1}$.  Thus by Proposition~\ref{lem:max}, $\Gamma$ is a string C-group.
Finally, to prove that the group is the collection of all even permutations in the wreath product, observe that the group $\Gamma_0$ gives the full symmetric group acting on the blocks, and the element $(\rho_0 \rho_1)^2$ is a product of two disjoint transpositions, each swapping two elements within a block.

 \end{proof}
 

\begin{lemma}\label{lem:Bl}
For each $r \geq 6$ the permutation group $Bl(r,k)$ given by the following graph is a string C-group isomorphic to  $ (S_2 \wr S_{r+2+\frac{k}{2}})^+ $,  for all $k \equiv 2 \pmod{4}$, with $k \geq 0$.
$$\includegraphics[scale=.5]{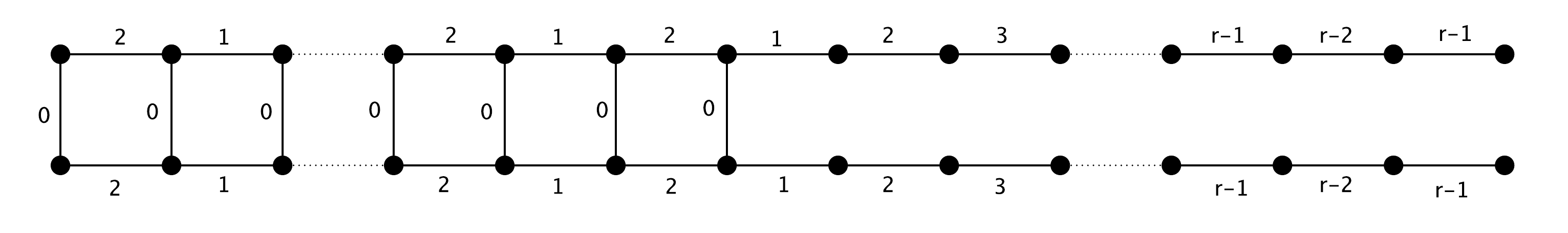}$$
\end{lemma}

We note that there are $1+\frac{k}{2}$ edges of label 0 in each graph.

\begin{proof}
The group $Bl(r,k)$ is obtained by applying the dual rank reduction of Lemma~\ref{lem:RR} to $Sh(r+2,k)$ and then again to the result.  This shows that $Bl(r,k)$ is a string C-group, and also that $Bl(r,k)$ is isomorphic to  $ (S_2 \wr S_{r+2+\frac{k}{2}})^+$.
\end{proof}


 \begin{lemma}\label{lem:P}
For each $r \geq 4$ the permutation group $P(r,k)$ given by the following graph is a string C-group isomorphic to $S_{r+2+k}$, for all odd $k \geq 0$, and
$$\includegraphics[scale=.5]{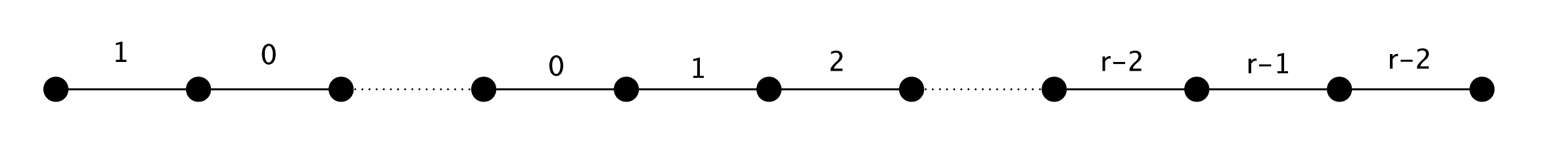}$$
for each $r \geq 4$ the permutation group $P(r,k)$ given by the following graph is a string C-group isomorphic to $S_{r+2+k}$, for all even $k \geq 0$, and
$$\includegraphics[scale=.5]{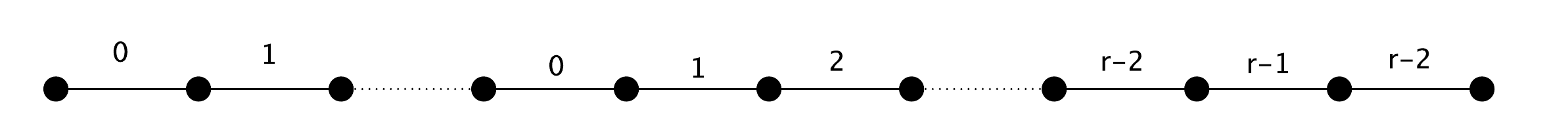}$$
\end{lemma}

There are $k+2$ edges of labels 0 or 1 in each graph.  Also, note that when $r=3$ this graph still provides a string C-group, see for example Theorem 4.5 of~\cite{DanielCPR} when $k$ is odd or Theorem 4.4 of~\cite{DanielCPR} when $k$ is even.

\begin{proof}
Let $\Gamma = \langle \rho_0, \ldots, \rho_{r-1} \rangle = P(r,k)$.   The group $\Gamma_0$ is a string C-group isomorphic to $S_2 \times S_{r+1}$ as it is a sesqui-extension of string C-group $FL(r-1,1)$ from Lemma~\ref{lem:FL}.  The group $\Gamma_{r-1}$ is a string C-group isomorphic to $S_2 \times S_{r+k}$ as it also is a sesqui-extension of string C-group $FL(r-1,k)$. Thus $\Gamma$ is isomorphic to $S_{r+2+k}$.
Finally, the group $\Gamma_{0,r-1}$ is isomorphic to $S_2 \times S_{r-1} \times S_2 $, which is maximal in $\Gamma_0$, and thus, by Proposition~\ref{lem:max}, $\Gamma$ is also a string C-group.
 \end{proof}


 \begin{lemma}\label{lem:Sp}
For each $r \geq 4$ the permutation group $Sp(r,k)$ given by the following graph is a string C-group isomorphic to $S_{r+2+\frac{k}{2}} \times S_{r+1+\frac{k}{2}}$  for $k \equiv 2 \pmod{4}$ with $k \geq 0$.  
$$\includegraphics[scale=.5]{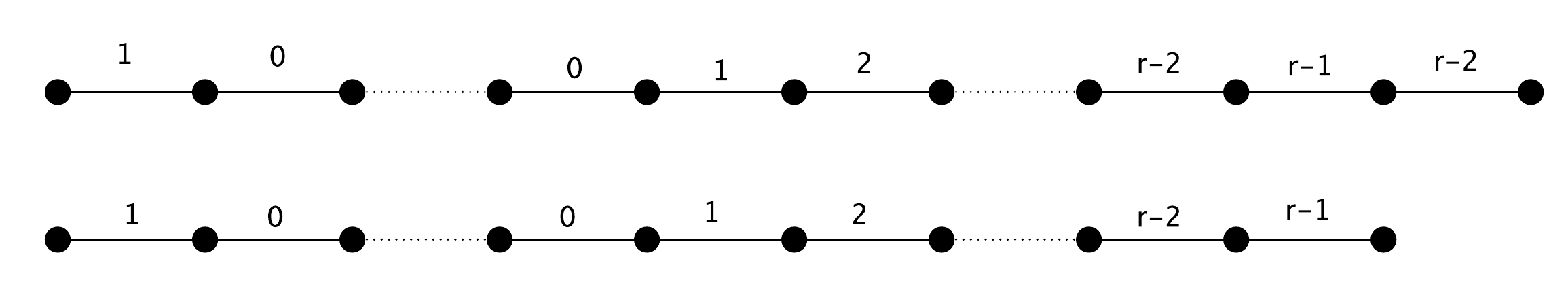}$$
\end{lemma}

We note that there are $k+4$ edges of labels 0 or 1 in each such graph.  

\begin{proof}
Let $\Gamma = \langle \rho_0, \ldots, \rho_{r-1} \rangle = Sp(r,k)$.  Then, $\Gamma = P \diamond Q$, where $P = P(r,\frac{k}{2})$ from Lemma~\ref{lem:P}, and $Q$ is $FL(r,\frac{k}{2})$ from Lemma~\ref{lem:FL}.  Furthermore, the facets of $P$ cover the facets of $Q$; as the facets of $P$ are isomorphic to $S_{r+\frac{k}{2}} \times S_2$ where as the facets of $Q$ are isomorphic to $S_{r+\frac{k}{2}}$.  Thus by Proposition~\ref{lem:MixMon}, $\Gamma \cong P \diamond Q$ is a string C-group.

To determine that $\Gamma$ is the full direct product, we consider the comix of $P$ and $Q$, $C= P \boxempty Q$.  We will write $C$ as generated by $\rho_i$, along with all the relations from both $P$ and $Q$.   Thus, in $C$, $(\rho_{r-1} \rho_{r-2})^3 = 1$ from $Q$ and $1 =  (\rho_{r-1} \rho_{r-2})^4$ from $P$, and thus in $C$,  $\rho_{r-1} = \rho_{r-2}$.   

Then, it follows that in $C$, $(\rho_{r-2} \rho_{r-3})^3 = 1 =  (\rho_{r-1} \rho_{r-3})^3 = (\rho_{r-1} \rho_{r-3})^2$, and thus $\rho_{r-1} = \rho_{r-3}$.  Furthermore in $C$, we know that $(\rho_{r-3} \rho_{r-2} \rho_{r-1})^5 = 1$ from $P$, and so for instance $\rho_{r-1}^5 = \rho_{r-1}^2=1$, and thus $\rho_{r-1}=\rho_{r-2} = \rho_{r-3}= 1$.

 Then, it will follow that all $\rho_i = 1$ in $C$.  For example in $C$, $(\rho_{r-3} \rho_{r-4})^3 = 1$ and thus in C $\rho_{r-4} =1 $.   This argument works for showing that in $C$, $1= \rho_2 = \rho_3 = \cdots = \rho_{r-1}$. 
 
 Finally, $(\rho_0 \rho_1 \rho_2)^{\frac{k}{2}+4} =1$  and $(\rho_0 \rho_1)^{\frac{k}{2}+3} =1$, and so in $C$, $\rho_0=\rho_1$.

We have showed that $ C= P \boxempty Q $ has size at most two, and therefore $P \diamond Q$ is the full direct product or an index two subgroup of the full direct product.  There are three index two subgroups of $S_{r+2+\frac{k}{2}} \times S_{r+1+\frac{k}{2}}$: namely $A_{r+2+\frac{k}{2}} \times S_{r+1+\frac{k}{2}}$, $S_{r+2+\frac{k}{2}} \times A_{r+1+\frac{k}{2}}$, and $(S_{r+2+\frac{k}{2}} \times S_{r+1+\frac{k}{2}})^+$.    As there are odd permutations in $P$, odd permutations in $Q$, and odd permutations in $P \diamond Q$, we can rule out all three cases and conclude that $\Gamma \cong S_{r+2+\frac{k}{2}} \times S_{r+1+\frac{k}{2}}$.

 \end{proof}


 \begin{lemma}\label{lem:Sm}
For each rank $r \geq 4$ the permutation group $Sm(r)$ given by the following graph is a string C-group isomorphic to $(S_r \times S_{r+3})^+ $.  
$$\includegraphics[scale=.5]{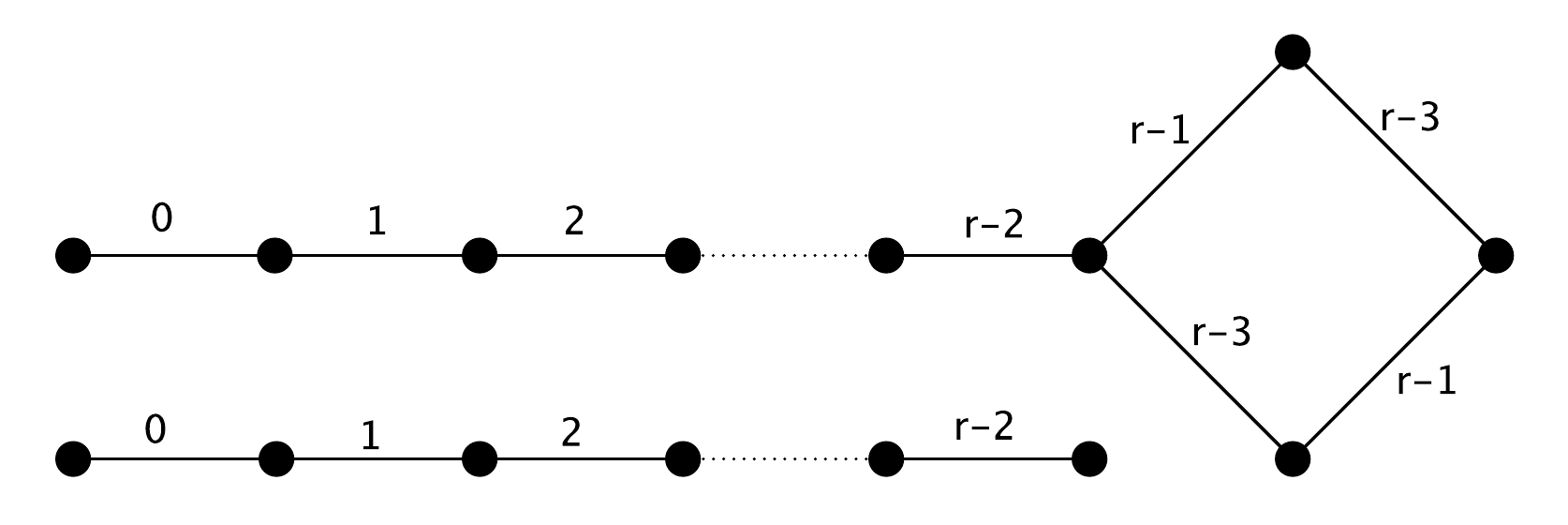}$$
\end{lemma}

\begin{proof}
Let $\Gamma = \langle \rho_0, \ldots, \rho_{r-1} \rangle = Sm(r)$.   We proceed by induction on the rank $r$, where the base case of $r=4$ can be verified using Magma.  The group $\Gamma_0$ is isomorphic to $Sm(r-1)$, and by induction we may assume that $\Gamma_0$ is a string C-group isomorphic to $(S_{r-1} \times S_{r+2})^+$.   Therefore, it can be seen that $\Gamma \cong (S_r \times S_{r+3})^+$, and it remains to show that $\Gamma$ is a string C-group.
The group $\Gamma_{r-1}$ is isomorphic to the sesqui-extension of a group $Sp(r-1,0)$ through $\rho_{r-3}$.  By parts (3) and (4) of Lemma~\ref{se}, we can see that $\Gamma_{r-1}$ is a string C-group.  Finally, $\Gamma$ is a string C-group by Lemma~\ref{lem:max} as $\Gamma_{0,r-1} \cong (S_{r-2} \times S_{r+1} \times S_2)^+$ which is maximal in $\Gamma_0$.  
\end{proof}


 \begin{lemma}\label{lem:Sy}
For each rank $r \geq 6$ the permutation group $Sy(r,k)$ given by the following graph is a string C-group isomorphic to $(S_{r+\frac{k}{2}} \times S_{r+3+\frac{k}{2}})^+ $ for all $k \equiv 2 \pmod{4}$, with $k \geq 0$, and  
$$\includegraphics[scale=.5]{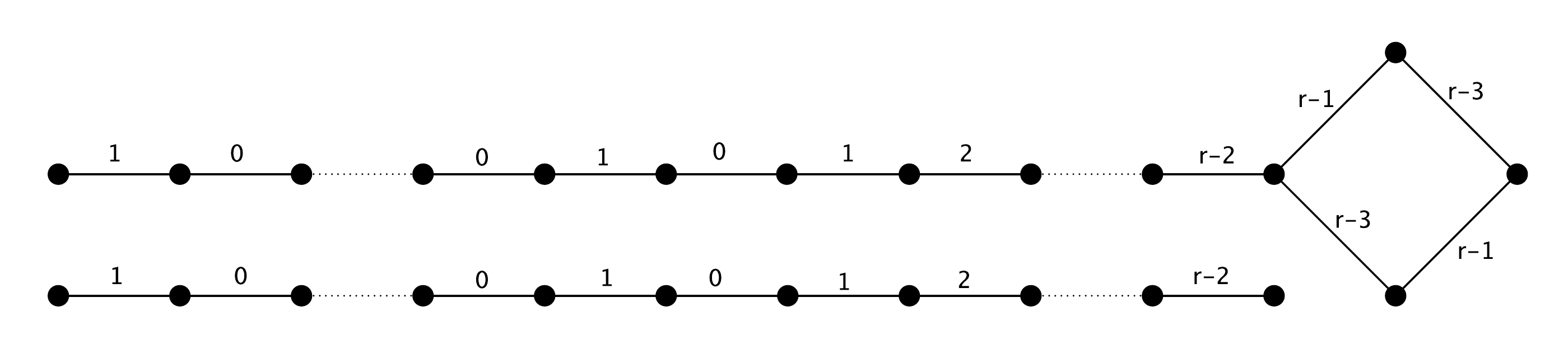}$$
for each rank $r \geq 6$ the permutation group $Sy(r,k)$ given by the following graph is a string C-group isomorphic to $(S_{r+\frac{k}{2}} \times S_{r+3+\frac{k}{2}})^+ $ for all $k \equiv 0 \pmod{4}$, with $k \geq 0$, and  
$$\includegraphics[scale=.5]{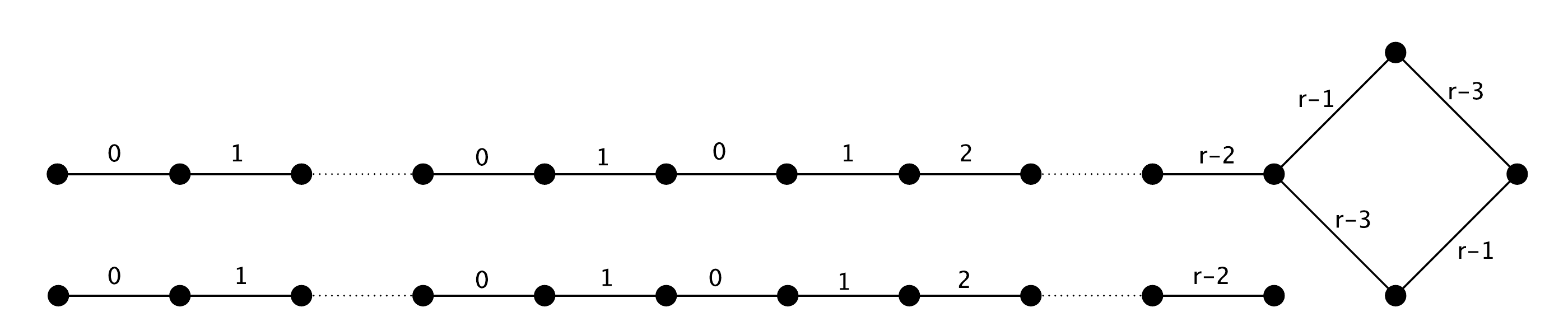}$$

\end{lemma}

\begin{proof}
This can be shown by induction on $k$.  When $k=0$ this representation is given in Lemma~\ref{lem:Sm}.  Assuming by induction that $Sy(r,k)$ is a string C-group, it follows that $Sy(r,k+2)$ is a string C-group by applying the dual rank reduction of Proposition~\ref{lem:RR} to $Sy(r+1,k)$; in order to apply the dual rank reduction, we need that $r+1 \geq 7$.
\end{proof}


 \begin{lemma}\label{lem:L(r,k)}
For each rank $r \geq 6$ the permutation group $L(r,k)$ given by the following graph is a string C-group isomorphic to $S_{r+3+k}$, for all odd $k \geq 0$, and 
$$\includegraphics[width=6in]{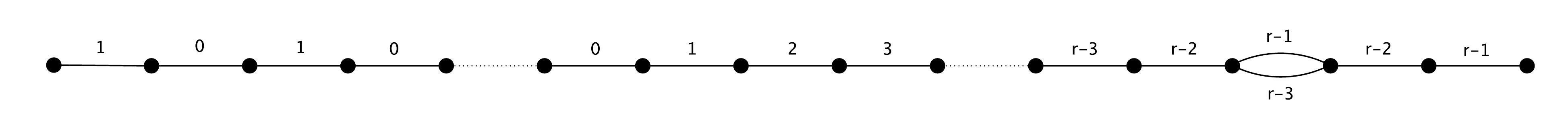}$$
for each rank $r \geq 6$ the permutation group $L(r,k)$ given by the following graph a string C-group isomorphic to $S_{r+3+k}$, for all even $k \geq 0$. 
$$\includegraphics[width=6in]{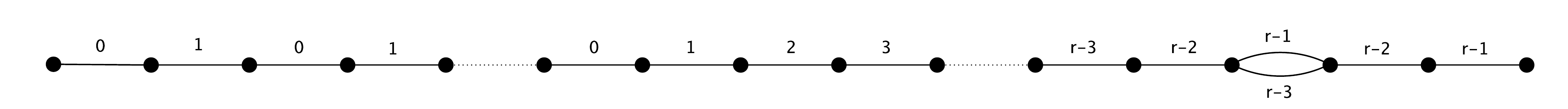}$$
 Furthermore $L_{r-1}$ is isomorphic to $S_{r+2+k}$.
\end{lemma}

To clarify, we note that there are $k+2$ edges of labels either 0 or 1 in each such graph.

\begin{proof}
This can be shown by induction on $k$.  When $k=0$ this representation was shown to be a string C-group for all $r\geq 6$ in~\cite{Extension_n-4}.  Assuming by induction that $L(r,k)$ is a string C-group, it follows that $L(r,k+1)$ is a string C-group by applying the rank reduction of Proposition~\ref{lem:RR} to $L(r+1,k)$.  The structure of $L_{r-1}(r,k)$ follows, again using Proposition~\ref{lem:RR}, from its relationship to $FL(r-1,0)$ from Lemma~\ref{lem:FL}.
\end{proof}


 \begin{lemma}\label{lem:M}
For each rank $r \geq 6$ the permutation group $M(r,k)$ given by the following graph is a string C-group, for all $k \equiv 2 \pmod{4}$, with $k \geq 0$, and  
$$\includegraphics[width=6in]{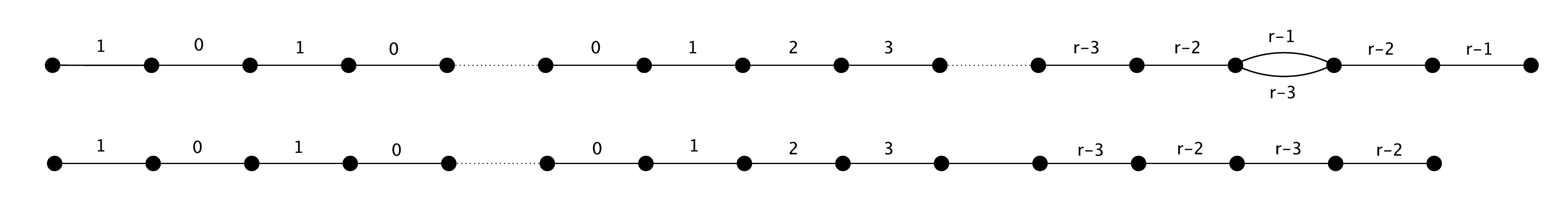}$$
for each rank $r \geq 6$ the permutation group $M(r,k)$ given by the following graph is a string C-group, for all $k \equiv 0 \pmod{4}$, with $k \geq 0$.  
$$\includegraphics[width=6in]{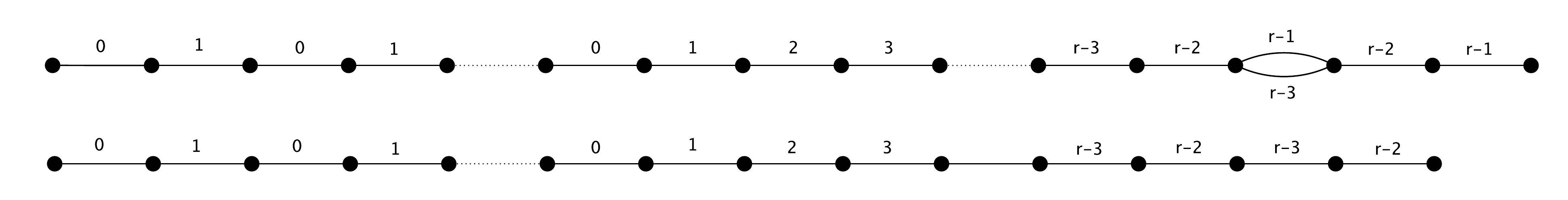}$$

\end{lemma}

We note that permutation degree of $M(r,k)$ is $2r+5+k$, and there are $k+4$ edges of labels either 0 or 1 in such a graph.

\begin{proof}
The group $M(r,k)$ is the mix of $L(r,k)$ with $L_{r-1}(r,k)$.  Thus $M(r,k)$ is a string C-group by Proposition~\ref{lem:MixFacet}. 
\end{proof}


 \begin{lemma}\label{lem:Sl(r,k)}
For each rank $r \geq 6$ the permutation group $Sl(r,k)$ given by the following graph is a string C-group isomorphic to $S_{2r+1+k}$ , for $k \equiv 2 \pmod{4}$ with $k \geq 0$.  
$$\includegraphics[width=6in]{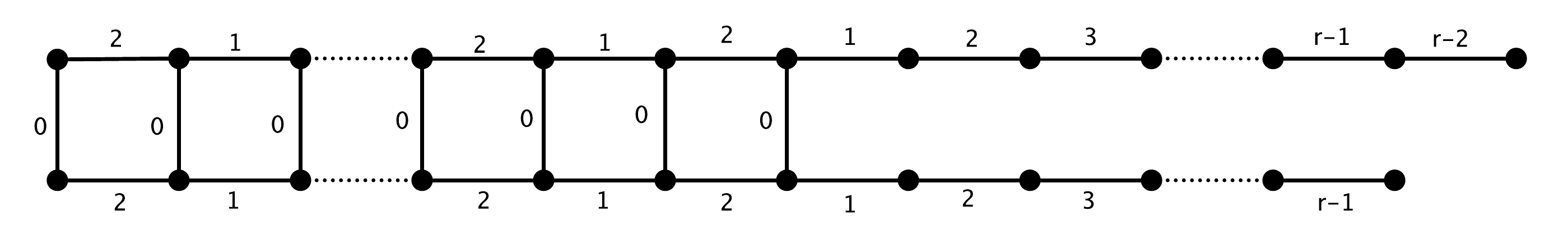}$$
\end{lemma}

We note that there are a total of $\frac{k}{2}+1$ edges of label 0 in such a graph.

\begin{proof}
Let $\Gamma = \langle \rho_0, \ldots, \rho_{r-1} \rangle = Sl(r,k)$.  
The group $\Gamma_0$ is a string C-group, as it is isomorphic to $Sp(r-1,\frac{k}{2})$ from Lemma~\ref{lem:Sp}, and thus $\Gamma_0$ is isomorphic to $S_{r+1+\frac{k}{2}} \times S_{r+\frac{k}{2}}$.   Since $\Gamma_0$ is maximal in  $S_{2r+1+k}$, we conclude that $\Gamma \cong S_{2r+1+k}$.

Let $Sh(r-1,k) = \langle a_0, \ldots, a_{r-2} \rangle$, from Lemma~\ref{lem:Sh}.  By part (2) of Lemma~\ref{se}, since $(a_{r-3} a_{r-2})^3 = 1$, we know that $\Gamma_{r-1} \cong Sh(r-1,k) \times S_2$, and thus $\Gamma_{r-1}$ is a string C-group isomorphic to $(S_2 \wr S_{r-1+\frac{k}{2}})^+ \times S_2$.  Notice that although $\Gamma_{r-1}$ is not transitive, it still has an imprimitive block structure, with blocks of size two, now with one more block, and also a block of size one.  It remains to show that $\Gamma_{r-1} \cap \Gamma_{0} \cong \Gamma_{0,r-1}$ which we will do by analyzing the orbits of $\Gamma_{r-1} \cap \Gamma_{0}$.

If $\alpha \in \Gamma_{r-1} \cap \Gamma_{0}$ then: $\alpha$ preserves the two orbits of $\Gamma_0$;  $\alpha$ preserves the three orbits of $\Gamma_{r-1}$, namely, it fixes the vertex of the graph which is incident to only an edge of label $r-1$, and it either fixes the vertex of the graph which is incident only to an edge of label $r-2$, or it it interchanges it with the other vertex on that edge of label $r-2$.  Finally, $\alpha$ preserves the block structure of $\Gamma_{r-1}$.  Thus, we can see $\alpha$ acting only on the blocks, since each block consists of two elements in different $\Gamma_0$ orbits.  Therefore, 
$\Gamma_{r-1} \cap \Gamma_{0} \leq (S_{r-1+\frac{k}{2}} \times S_2)$. It is easy to check that $\Gamma_{0,r-1} \cong (S_{r-1+\frac{k}{2}} \times S_2)$, and thus $\Gamma$ is a string C-group.

 \end{proof}


\section{Odd degree and high rank}\label{sec4}

In this section we deal with alternating groups of odd permutation degree, represented as string C-groups of rank at least seven.

\begin{theorem}\label{th:Steve}

For each $r \geq 7$ the permutation group $S(r,k)$ given by the following graph is a string C-group isomorphic to $A_{2r+1+k}$ for all $k \equiv 2 \pmod{4}$, with $k \geq 0$.  
$$\includegraphics[scale=.5]{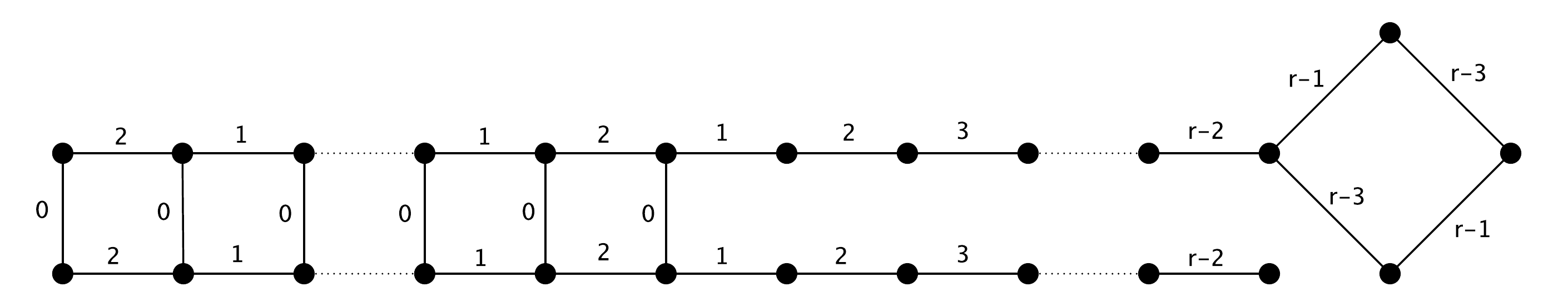}$$

\end{theorem}

\begin{proof}
Let $\Gamma = \langle \rho_0, \ldots, \rho_{r-1} \rangle = S(r,k)$.  The group $\Gamma_0 \cong Sy(r-1,k)$ is a string C-group isomorphic to $(S_{r-1+\frac{k}{2}} \times S_{r+2+\frac{k}{2}})^+$ by Lemma~\ref{lem:Sy}.
The group $\Gamma_{r-1} \cong P \diamond Q$ where $P = Sl(r-1,k)$ is a string C-group isomorphic to $S_{2r-1+k}$, and $Q$ is a single involution extending $\rho_{r-3}$.  Since $\Gamma_{r-1}$ only contains even permutations, by parts (3) and (4) of Lemma~\ref{se}, $\Gamma_{r-1}$ is a string C-group isomorphic to $S_{2r-1+k}$.
Finally, by Lemma~\ref{lem:Sp}, the group $\Gamma_{0,r-1}\cong Sp(r-2,k)$ is isomorphic to $S_{r+\frac{k}{2}} \times S_{r-1+\frac{k}{2}}$, which is maximal in $\Gamma_0$.  Thus $\Gamma$ is a string C-group by Lemma~\ref{lem:max}.

It is clear that $\Gamma$ is a subgroup of $A_{2r+1+k}$.  The main theorem of~\cite{LPMax} shows that $(S_{r-1+\frac{k}{2}} \times S_{r+2+\frac{k}{2}})^+$ is maximal in $A_{2r+1+k}$ and thus $\Gamma \cong A_{2r+1+k}$.

\end{proof}

\begin{theorem}\label{th:Barry}

For each $r \geq 7$ the permutation group $B(r,k)$ given by the following graph is a string C-group isomorphic to $A_{2r+3+k}$ for all $k \equiv 2 \pmod{4}$, with $k \geq 0$.  
  
$$\includegraphics[scale=.44]{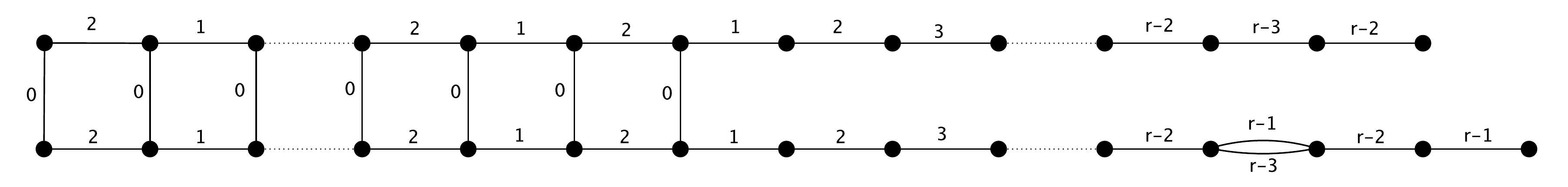}$$

\end{theorem}

\begin{proof}
Let $\Gamma = \langle \rho_0, \ldots, \rho_{r-1} \rangle = B(r,k)$.

First let us show that $\Gamma$ isomorphic to $A_{2r+3+k}$.  This can be checked using {\sc Magma} for $B(7,2)$, $B(7,6)$, $B(8,2)$, and $B(9,2)$ where $n= 2r+3+k \leq 23$.  For all the other cases, $n > 24$, and we can use Corollary 1.2 of~\cite{Maroti}.  In these cases we know that, $ | \Gamma | < 2^{n}$ or $\Gamma \cong A_{n}$.  The group $\Gamma_0$ contains a symmetric group acting on one its orbits, and thus $ | \Gamma |   > (\frac{n+1}{2})!$; however for all $n \geq 25$,  $(\frac{n+1}{2})! > 2^n$, and so $\Gamma \cong A_n$.

The group $\Gamma_0$ is a string C-group $M(r-1,k)$ from Lemma~\ref{lem:M}.  The group $\Gamma_{r-1}$ is a string C-group $Bl(r-1,k)$ by Lemma~\ref{lem:Bl}.  Note that although $\Gamma_{r-1}$ is not transitive, it still has an imprimitive block structure, with blocks of size two, and also a block of size one. 

It remains to show that $\Gamma_{r-1} \cap \Gamma_{0} \cong \Gamma_{0,r-1}$ which we will do by analyzing the orbits of $\Gamma_{r-1} \cap \Gamma_{0}$.   If $\alpha \in \Gamma_{r-1} \cap \Gamma_{0}$ then: $\alpha$ preserves the two orbits of $\Gamma_0$; $\alpha$ preserves the block structure of $\Gamma_{r-1}$, and $\alpha$ preserves the fixed point of $\Gamma_{r-1}$.  Thus $\Gamma_{r-1} \cap \Gamma_{0}$ can be seen acting on the $r+1+\frac{k}{2}$ blocks of size two, and is isomorphic to a subgroup of $S_{r+1+\frac{k}{2}}$.  

Finally, the group $\Gamma_{0,r-1} \cong P \diamond P$ where $P$ is a string C-group $R(r-2,\frac{k}{2})$ isomorphic to $S_{r+1+\frac{k}{2}}$, by Lemma~\ref{lem:R}, and thus $\Gamma$ is a string C-group.

\end{proof}

\begin{corollary}\label{lem:OddDone}
For each $r \geq 7$, and each $n \geq 2r+1$, there is a string C-group representation of $A_n$ of rank $r$.
\end{corollary}

\begin{proof}
When $n=2r+1$ this follows from Theorem 7.2 of~\cite{Alt2}.  When $n=2r+3+4j$ for some integer $j$, it follows from Theorem~\ref{th:Steve}.  Finally, when $n=2r+5+4j$ for some integer $j$, it follows from Theorem~\ref{th:Barry}.  
\end{proof}


\section{Even degree and high rank}\label{sec5}

In this section we deal with alternating groups of even permutation degree, represented as string C-groups of rank  at least seven.   

 \begin{theorem}\label{lem:D}
For each rank $r \geq 6$ the permutation group $D(r,k)$ given by the following graph is a string C-group isomorphic to $A_{2r+2+k}$, for all $k \equiv 2 \pmod{4}$, with $k \geq 0$, and  
$$\includegraphics[width=6in]{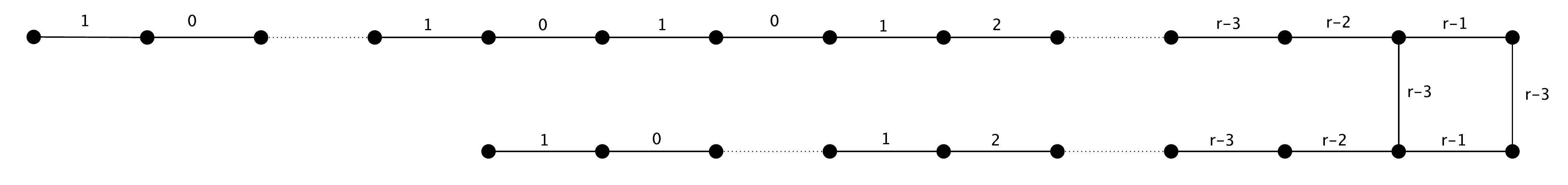}$$
for each rank $r \geq 6$ the permutation group $M(r,k)$ given by the following graph is a string C-group isomorphic to $A_{2r+2+k}$, for all $k \equiv 0 \pmod{4}$, with $k \geq 0$.  
$$\includegraphics[width=6in]{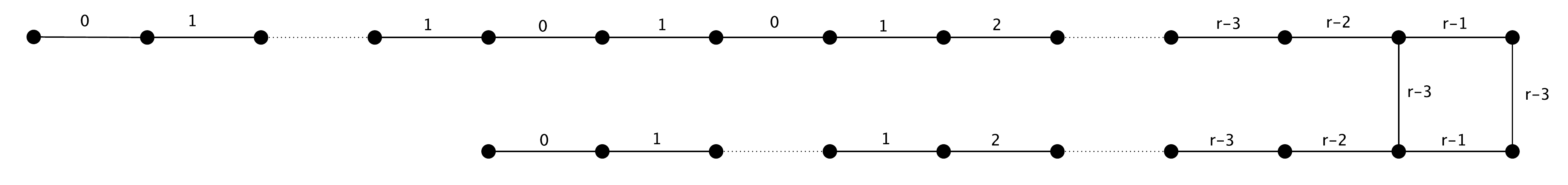}$$
\end{theorem}
To clarify this notation we both include an example $D(7,4)$ and remark that there are $k+4$ edges with labels 0 or 1 in each such graph.

\begin{figure}[h]
$$\includegraphics[width=6in]{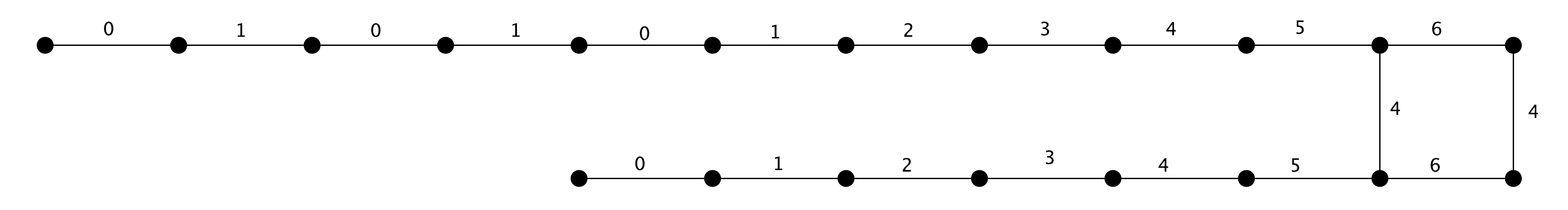}$$
\caption{ The group $D(7,4)$ \label{D(7,4)} }
\end{figure}

\begin{proof}
This is proved by induction on $k$.  When $k=0$, $D(r,k)$ was shown to be a string C-group isomorphic to $A_{2r+2}$ in Theorem 7.1 of~\cite{Alt2}.  To see that $D(r,k+2)$ is a string C-group isomorphic to $A_{2r+2+k+2}$, we notice that you obtain $D(r,k+2)$ from $D(r+1,k)$ using the rank reduction of Proposition~\ref{lem:RR}, which is assumed to be a string C-group isomorphic to $A_{2r+2+k+2}$ by induction.  We can apply the rank reduction to $D(r+1,k)$ as long as $r+1 \geq 7$.
\end{proof}

\begin{corollary}
\label{lem:EvenDone} For all $r \geq 7$, the alternating group $A_n$ can be represented as a rank $r$- string if $n$ is even and $n \geq 2r+2$.	
\end{corollary}	

\section{Low Ranks and Extension Construction} \label{sec6}

In this section we deal with alternating groups represented as rank 4, 5, or 6 string C-groups.

In~\cite{Nuzhin} it was showed that $A_n$ is generated by three involutions two of which commute if and only if $n \not \in \{3,4,6,7,8 \}$.  Additionally, in~\cite{DanielCPR}, permutation representations for the string C-groups for each of these $A_n$ were provided.  Thus it is known exactly which $A_n$ are rank 3 string C-groups.  

\begin{proposition}\label{th:r3} The alternating group $A_n$ can be represented as a rank 3 string C-group if and only if $n=5$ or $n \geq 9$.
\end{proposition}

In order to show which $A_n$ can be represented as a rank $r$ string C-group (for $r = 4,5,6$) we mainly rely on the following construction, which is similar to one of Pellicer~\cite{PExt} or of Schulte~\cite{SExt}.

\begin{definition} \label{tail} 
Given an sggi $\Gamma$ whose permutation representation graph $X$ has vertices $[1,\ldots,m]$ and edges labeled $[1,\ldots,r-1]$, we define $\Gamma^t$ (for each $t>0$) as the group generated by involutions (with one more generator) whose permutation representation graph has vertices $[1,\ldots, m+t]$ and is obtained by adjoining an alternating path of length $t$ of edges labeled 0,1 to the vertex $m$ of $X$.
\end{definition}
\begin{figure}
$$\includegraphics[scale=.4]{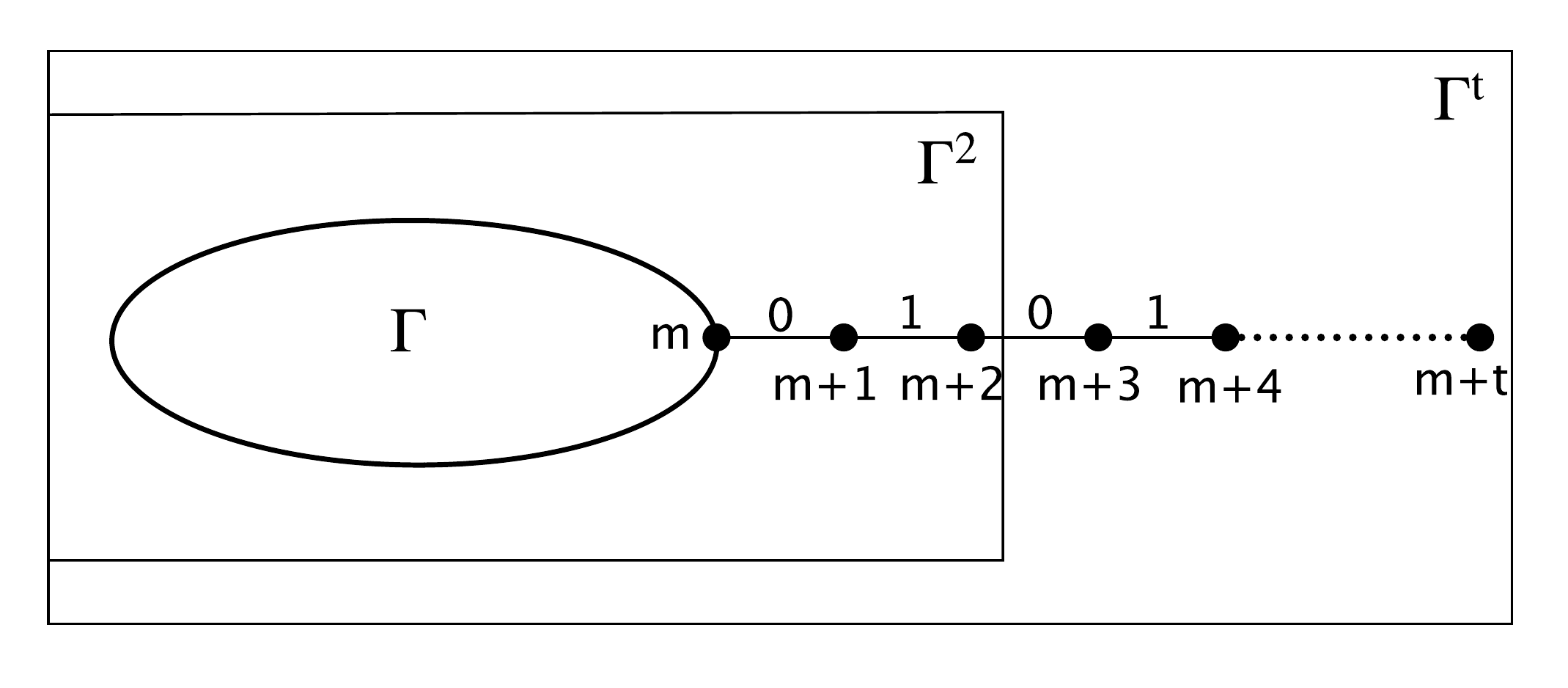}$$  
\caption{The group $\Gamma$ is an string group generated by involutions.  Then $\Gamma^t$ is a group with one more generator and larger permutation degree. \label{fig:Construction} }
\end{figure}

As a matter of notation, recall that for an sggi $G = \langle \rho_0, \ldots, \rho_{r-1} \rangle$ and an integer $k$ we define the groups $G_{\leq k}=\langle \rho_0, \ldots, \rho_{k} \rangle$.

\begin{proposition}
\label{lem:add4}
Let $\Gamma$ be an sggi acting on the points $[1,\ldots,m]$, and for each positive integer $t$ construct $\Gamma^t$ as described above.  If there is an integer $b \geq 2$ such that $\Gamma^b$ is a string C-group, and for each $2 \leq k \leq r-2$, the group $\Gamma^b_{\leq k}$ acts as a symmetric group on Orbit($\Gamma^b_{\leq k},m)$, then for each $t > b$, the group $\Gamma^t$ is also a string C-group.
\end{proposition}

\begin{proof}
Assume that for some $b \geq 2$, $\Gamma^b = \langle \rho_0, \ldots, \rho_{r-1} \rangle$ is a string C-group.  Then, by the commuting relationship between the generators, only the elements $\rho_0$ and $\rho_1$ can have the point $m$ in their support, and thus $m$ is of degree 2 in the natural CPR graph of $\Gamma^b$.   Let $\Gamma^t = \langle \rho_0, \ldots, \rho_{r-1} \rangle$.   To show that $\Gamma^t$ is a string C-group we rely on part (b) of Proposition 2E16 of~\cite{arp}.  There is an isomorphism $\phi$ that maps generators of $\Gamma^t_0$ to generators of $\Gamma^b_0$, and thus $\Gamma^t_0$ is also a string C-group.  It remains to show that for each $k=0,\ldots,r-2$ that $\Gamma^t_0 \cap \Gamma^t_{\leq k} = \langle \rho_1, \ldots, \rho_{k} \rangle$.  

We first deal with small values of $k$. If $k=0$, it is clear by construction that $\rho_0 \not \in \Gamma^t_0$ and thus $\Gamma^t_0 \cap \langle \rho_0 \rangle = \langle 1 \rangle$.  Let $k=1$ and $\alpha \in  \Gamma^t_0 \cap \langle \rho_0, \rho_1 \rangle$.  We assume, without loss of generality, that $\alpha$ has an even number of factors of $\rho_1$ and thus (as it is in $ \Gamma^t_0$) fixes all points greater than or equal to $m$.  Since $\alpha \in  \langle \rho_0, \rho_1 \rangle$ with an even number of factors of $\rho_1$ it also fixes all points less than $m$.  Therefore $\alpha$ is identity and 
 $\Gamma^t_0 \cap \langle \rho_0, \rho_1 \rangle = \langle \rho_1 \rangle$.

Now assume that $k>2$ and let $\alpha \in \Gamma^t_0 \cap \Gamma^t_{\leq k}$ again fixing all points greater than or equal to $m$.  Consider $\alpha$ as a word in $\Gamma^t_0$. Using the isomorphism $\phi$, we get a word $\alpha' \in  \Gamma^b_0$ by changing all the factors of $\rho_j$ in $\alpha$ to $\rho_j$.  Under the conditions of the proposition, we assume that the group $\Gamma^b_{\leq k}$ acts as a symmetric group on Orbit($\Gamma^b_{\leq k},m)$.  Additionally, the action of $\alpha$ on all other orbits Orbit($\Gamma^t_{\leq k},p)$ is identical to the action of $\alpha'$ on Orbit($\Gamma^b_{\leq k},p)$, and thus any permutation in $\Gamma^t_{\leq k}$ that fixes the points greater than or equal to $m+b+1$ can also be written as a word in $\Gamma^b_{\leq k}$.     Thus the permutation associated with the word $\alpha'$ can also be written as a word in $\Gamma^b_{\leq k}$.

Finally, as $\Gamma^b$ is a string C-group, $\alpha' \in \Gamma^b_{\leq k}$ and $\alpha' \in \Gamma^b_0$ implies that 
$ \alpha' \in \langle \rho_1, \ldots, \rho_{k} \rangle$.  Again using the isomorphism $\phi$ we get $\alpha \in \langle \rho_1, \ldots, \rho_{k} \rangle$ as required.

\end{proof}

If one can find appropriate groups $\Gamma^b$ which fit the conditions of the proposition above, then these can be used to build new string C-groups of the same rank for groups of larger permutation degree.  

  In particular, for each rank $r \in \{4,5,6\}$, the proof that $A_n$ can be represented as a rank $r$ string C-group (for all large $n$) is done by giving examples of string C-groups $\Gamma^2$ (see Figures~\ref{rank4ex}, \ref{rank5ex}, and \ref{rank6ex}) that satisfy the conditions of Proposition~\ref{lem:add4}, and such that for some $j$, the group $\Gamma^j$ is isomorphic to $A_n$.  These examples then yield families of string C-groups $\Gamma^{j+4k} \cong A_{n+4k}$.  
  Thus, for each rank, we need to find an example of $\Gamma^{t} \cong A_{n}$ for each value of $n \mod 4$.

  \begin{theorem}\label{lem:low}
  The alternating group $A_n$ has a rank 4 string C-group representation if and only if $n=9$, $n=10$, or $n \geq 12$.  The alternating group $A_n$ has a rank 5 string C-group representation if and only if $n=10$, or $n \geq 12$.  Finally, The alternating group $A_n$ has a rank 6 string C-group representation if and only if $n=11$ or $n \geq 13$.  
  
  \end{theorem}  
  
  \begin{proof}
  All of the string C-group representations of $A_n$ for $n \leq 10$ were classified in~\cite{high-rank-alternating}, with $n=11,12,13,$ and $14$ subsequently classified in~\cite{LMAlg}.    It remains to show that, for all $n \geq 15$, there are rank $4$, rank $5$, and rank $6$ string C-group representations of $A_n$.
  
  For each rank, by listing the generators $\rho_i$, we now give four examples of string C-groups $\Gamma^2$  that fit the conditions of Proposition~\ref{lem:add4}, and show which value of $t$ gives $\Gamma^{t} \cong A_{n}$.  To clarify the notation, we include an image of the first family below.
\begin{figure}[h]
$$\includegraphics[width=6in]{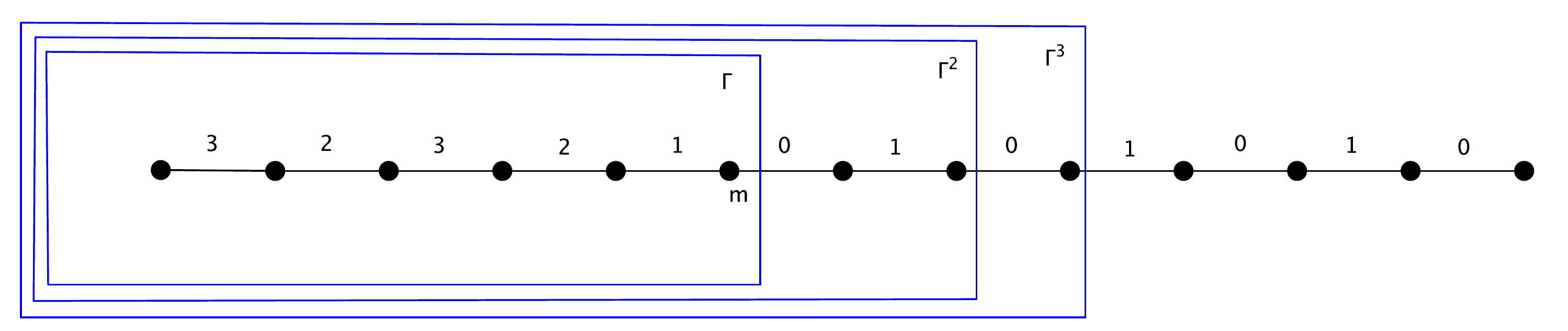}$$
\caption{$\Gamma$, $\Gamma^2$, $\Gamma^3$, and $\Gamma^7$ from Family 1 of rank 4 string C-group representations.   The group~$\Gamma^3$ is isomorphic to $A_{9}$, and the group~$\Gamma^7$ is isomorphic to $A_{13}$.  \label{fig:Family11} }
\end{figure}

\begin{figure}[h]
\begin{tabular}{ll}
\begin{tabular}{l}
Family 1: $\Gamma^{3+4k} \cong A_{9+4k}$.  \\
$\rho_0=(6,7).$\\
$\rho_1=(5,6)(7,8).$\\
$\rho_2=(2,3)(4,5).$\\
$\rho_3=(1,2)(3,4).$\\
\\
\end{tabular}
&
\begin{tabular}{l}
Family 2: $\Gamma^{3+4k} \cong A_{18+4k}$. \\
$\rho_0=(15,16).$\\
$\rho_1=(4,5)(6,7)(14,15)(16,17).$\\
$\rho_2=(1,2)(3,4)(6,8)(9,10)(11,12)(13,14).$\\
$\rho_3=(2,3)(4,6)(5,7)(8,9)(10,11)(12,13).$\\
\\
\end{tabular} 
\\
\begin{tabular}{l}
Family 3: $\Gamma^{4+4k} \cong A_{15+4k}$. \\
$\rho_0=(11,12).$\\
$\rho_1=(3,4)(5,7)(6,11)(10,9)(12,13).$\\
$\rho_2=(2,3)(4,6)(5,8)(7,10).$\\
$\rho_3=(1,2)(3,5)(4,7)(10,9).$\\
\end{tabular}
&
\begin{tabular}{l}
Family 4: $\Gamma^{4+4k} \cong A_{16+4k}$. \\
$\rho_0=(12,13).$\\
$\rho_1=(3,4)(5,7)(6,9)(10,12)(13,14).$\\
$\rho_2=(2,3)(4,6)(5,8)(7,10).$\\
$\rho_3=(1,2)(3,5)(4,7)(8,11).$\\
\end{tabular}
\end{tabular}
\caption{Rank groups $\Gamma^2$ satisfying the conditions of Lemma~\ref{lem:add4} \label{rank4ex}}
\end{figure}


\begin{figure}[h]
\begin{tabular}{ll}
\begin{tabular}{l}
Family 1: $\Gamma^{3+4k} \cong A_{13+4k}$. \hspace{1cm} \\
$\rho_0=(10, 11).$\\
$\rho_1=(9, 10)(11, 12).$\\
$\rho_2=(4, 5)(8, 9).$\\
$\rho_3=(1, 2)(3, 4)(5, 6)(7, 8).$\\
$\rho_4=(2, 3)(6, 7).$\\
\\
\end{tabular}
&
\begin{tabular}{l}
Family 2: $\Gamma^{3+4k} \cong A_{14+4k}$\\
$\rho_0= (11, 12).$\\
$\rho_1=(10, 11)(12, 13).$\\
$\rho_2=(5, 6)(9, 10).$\\
$\rho_3=(2, 3)(4, 5)(6, 7)(8, 9).$\\
$\rho_4= (1, 2)(3, 4)(5, 6)(7, 8).$\\
\\
\end{tabular} 
\\
\begin{tabular}{l}
Family 3: $\Gamma^{3+4k} \cong A_{15+4k}$. \\
$\rho_0=    (12, 13).$\\
$\rho_1=    (9, 12)(13, 14).$\\
$\rho_2=(3, 4)(5, 7)(6, 9)(10, 11).$\\
$\rho_3=(2, 3)(4, 6)(5, 8)(7, 10).$\\
$\rho_4= (1, 2)(3, 5)(4, 7)(10, 11).$\\
\end{tabular}
&
\begin{tabular}{l}
Family 4: $\Gamma^{3+4k} \cong A_{12+4k}$. \\
$\rho_0= (9, 10).$\\
$\rho_1=    (8, 9)(10, 11).$\\
$\rho_2=    (1, 2)(3, 4)(5, 6)(7, 8).$\\
$\rho_3=(2, 3)(6, 7).$\\
$\rho_4=(3, 5)(4, 6).$\\
\end{tabular}
\end{tabular}
\caption{Rank 5 groups $\Gamma^2$ satisfying the conditions of Lemma~\ref{lem:add4} \label{rank5ex}}
\end{figure}


\begin{figure}[h]
\begin{tabular}{ll}
\begin{tabular}{l}
Family 1: $\Gamma^{3+4k} \cong A_{17+4k}$. \hspace{1cm}  \\
$\rho_0=(14, 15).$\\
$\rho_1=(13, 14)(15, 16).$\\
$\rho_2= (6, 7)(12, 13).$\\
$\rho_3=(2, 4)(5, 6)(7, 8)(11, 12).$\\
$\rho_4=(1, 2)(3, 5)(8, 9)(10, 11).$\\
$\rho_5=(1, 3)(9, 10).$\\
\\
\end{tabular}
&
\begin{tabular}{l}
Family 2: $\Gamma^{3+4k} \cong A_{18+4k}$.\\
$\rho_0=(15,16).$\\
$\rho_1=  (14, 15)(16, 17).$\\
$\rho_2= (9, 12)(11, 14).$\\
$\rho_3= (3, 4)(5, 7)(6, 9)(8, 11).$\\
$\rho_4=(2, 3)(4, 6)(5, 8)(7, 10).$\\
$\rho_5= (1, 2)(3, 5)(4, 7)(10, 13).$\\
\\
\end{tabular} 
\\
\begin{tabular}{l}
Family 3: $\Gamma^{3+4k} \cong A_{15+4k}$. \\
$\rho_0=(12, 13).$\\
$\rho_1=(11, 12)(13, 14).$\\
$\rho_2=(2, 4)(5, 8)(6, 9)(10, 11).$\\
$\rho_3=(1, 2)(3, 6)(5, 8)(7, 10).$\\
$\rho_4=(2, 5)(3, 7)(4, 8)(6, 9).$\\
$\rho_5=(1, 3)(2, 6)(4, 9)(5, 8).$\\
\end{tabular}
&
\begin{tabular}{l}
Family 4: $\Gamma^{3+4k} \cong A_{16+4k}$.\\
$\rho_0= (13, 14).$\\
$\rho_1=(12, 13)(14, 15).$\\
$\rho_2= (8, 10)(11, 12).$\\
$\rho_3=(2, 3)(4, 6)(5, 8)(9, 11).$\\
$\rho_4=(1, 2)(3, 5)(4, 7)(6, 9).$\\
$\rho_5= (2, 4)(3, 6).$\\
\end{tabular}
\end{tabular}
\caption{Rank 6 groups $\Gamma^2$ satisfying the conditions of Lemma~\ref{lem:add4} \label{rank6ex}}
\end{figure}


First let us prove that $\Gamma^{t} \cong A_{n}$ for all given families.  To do this, we rely on the fact that if $\Gamma^{t}$ is primitive subgroup of $S_n$ and $\Gamma^{t}$ contains a 3-cycle then  
$\Gamma^{t}  \geq A_n$.  All twelve given families of groups $\Gamma^t$ are constructed so that for all $t \geq 2$
$$(\rho_2 \rho_1 \rho_0 \rho_1)^2$$
 is a 3-cycle.  Thus we will only need to show why the groups are primitive. 
  All twelve given families of groups $\Gamma^t$ are also constructed so that $\langle \rho_0, \rho_1, \rho_2\rangle$ acts as a symmetric group on the orbit of the point $m$.  Therefore, once $k \geq 5$, in all cases $\Gamma^t$ will contain a symmetric group acting on more than half its permutation degree, and thus cannot be imprimitive.   The isomorphisms between $\Gamma^t$ and $A_n$ have been found in \textsc{Magma} for $k \leq 4$.  Therefore, all the given families of groups yield alternating groups, and using Proposition~\ref{lem:add4}, they have been shown to be string C-groups, with the base case of $\Gamma^2$ checked in  \textsc{Magma}.

\end{proof}

We note that for the remaining values of $t$ each of these groups yields a string C-group representation of a symmetric group, as either $\rho_0$ or $\rho_1$ will be odd, and $\Gamma$ will still contain an alternating group.

\section{Main Theorem}\label{sec7}

In this section, we put together all of our results to summarize the relationship between alternating groups and string C-groups.

\begin{theorem}
For all ranks $r \geq 3$, and all $n \geq 2r+1$, if $n \geq 12$, the alternating group $A_n$ has a representation as a rank $r$ string C-group.
\end{theorem}

\begin{proof}
The small values of $n=12$, $n=13,$ and $n=14$ follow from~\cite{LMAlg}.  Ranks 3, 4, 5, and 6, were considered in Proposition~\ref{th:r3} and Theorem \ref{lem:low}.  For $r \geq 6$, when $n$ is even this follows from Corollary~\ref{lem:EvenDone}, and for $r \geq 7$ when $n$ is odd it follows from Corollary~\ref{lem:OddDone}.

\end{proof}

\begin{corollary}
For each rank $r$ it is known exactly which alternating groups can be represented as a string C-group of that rank.  Similarly, for each alternating group $A_n$, the exact set of ranks of string C-groups for which it can be represented is known.
\end{corollary}

\begin{proof}
This follows from the previous theorem in addition to Theorem 1.1 of~\cite{HighestAn}, which states that the highest rank of a string C-group representation of $A_n$ is $3$ if $n=5$, $4$ if $n=9$, $5$ if $n=10$, $6$ if $n=11$ and $\frac{(n-1)}{2}$ if $n \geq 12$.  Moreover, if $n=3, 4, 6, 7$, or $8$, the group $A_n$ is not a string C-group.
 
\end{proof}

\begin{corollary}
The group $A_{11}$ is the only alternating group that has string C-group representations of two ranks $r_1$ and $r_2$, but not all ranks $r_i$ in between $r_1$ and $r_2$.
\end{corollary}


\section{Acknowledgements}
The computations in this paper were completed using Magma \cite{Magma}.  I would like to thank Egon Schulte, Barry Monson, Daniel Pelicer, and Gabe Cunningham for insightful conversations that led to various improvements in this paper.
I would like to thank Dimitri Leemans and Maria Elisa Fernandes for their support throughout this project. Also I would like to thank Dimitri for first suggesting this problem in 2012; this would not have been possible without them.

\bibliographystyle{plain}

\end{document}